\documentclass[11pt, reqno]{amsart}
\usepackage{lmodern}
\usepackage{amsmath, amsthm, amssymb, amsfonts}
\usepackage[normalem]{ulem}
\usepackage{hyperref}
\usepackage{multirow}
\usepackage{bm}
\usepackage{extarrows}
\usepackage[utf8]{inputenc}
\usepackage{tikz}
\usepackage{tikz-cd}
\usepackage{hhline}
\usetikzlibrary{positioning,quotes}
\usetikzlibrary{arrows,shapes,automata,backgrounds,petri}

\usepackage{verbatim}
\usepackage{caption}
\usepackage{mathtools}

\usepackage{tikz-cd}

\theoremstyle{plain}
\newtheorem{thm}{Theorem}[section]

\newtheorem{lem}[thm]{Lemma}

\newtheorem{prop}[thm]{Proposition}

\theoremstyle{definition}

\theoremstyle{remark}
\newtheorem{rmk}[thm]{Remark}

\newcommand{\BC}{{\mathbb{C}}}

\newcommand{\BP}{{\mathbb{P}}}
\newcommand{\BQ}{{\mathbb{Q}}}
\newcommand{\BR}{{\mathbb{R}}}

\newcommand{\BZ}{{\mathbb{Z}}}

\newcommand{\CA}{{\mathcal A}}

\newcommand{\CH}{{\mathcal H}}

\DeclareFontFamily{OT1}{rsfs}{}
\DeclareFontShape{OT1}{rsfs}{n}{it}{<-> rsfs10}{}
\DeclareMathAlphabet{\curly}{OT1}{rsfs}{n}{it}

\newcommand\Spec{\operatorname{Spec}}

\newcommand{\Pic}{\mathop{\rm Pic}\nolimits}

\title{Mixed Hodge structures of cluster varieties of dimension 3}
\date{\today}

\author{Yuhang Zhang}
\address{School of Mathematical Sciences,
Key Laboratory of Intelligent Computing and Applications (Ministry of Education), 
Tongji University, Shanghai 200092, China}
\email{2230939@tongji.edu.cn}

\author{Zili Zhang}

\email{zhangzili@tongji.edu.cn}
\begin{document}

\begin{abstract}
   We calculate the mixed Hodge numbers of smooth 3-dimensional cluster varieties and show that they are of mixed Tate type. We also study the mixed Hodge structures of the cohomology and intersection cohomology groups of some singular cluster varieties. 
\end{abstract}

\maketitle
\setcounter{tocdepth}{1}
\tableofcontents
\section{Introduction}

Cluster algebras are certain $\BC$-algebras defined by $(n+m)\times n$ matrices $\widetilde{B}$, where the top $n\times n$ are anti-symmetric; see Section \ref{2} for details. The matrix $\widetilde{B}$ can also be viewed as the exchange matrix of a decorated directed graph $\Gamma$ with $n+m$ vertices and an edge $i\to j$ of weight $\widetilde{B}_{ij}$ if and only if $\widetilde{B}_{ij}>0$. The vertices $1,\cdots,n$ are called mutable vertices (marked as circles) and $n+1,\cdots,n+m$ are called frozen (marked as squares). By abuse of notation, we also say cluster varieties are defined by the graphs. Cluster algebras are not guaranteed to be finitely generated, but if they are, their affine spectra are called cluster varieties. If the graph $\Gamma$ is acyclic, the associated cluster variety is $n+m$ dimensional. In \cite{LS}, a more general finiteness condition, called the Louise property, is introduced, and it is shown that if matrix $\widetilde{B}$ is full rank and satisfies the Louise property, the associated cluster varieties are always smooth of dimension $n+m$.

For an algebraic variety $X$, the cohomology $H^*(X)$ carries a mixed Hodge structure naturally. Let $H^{k,(p,q)}(X)$ be the $(p,q)$-component of $H^k(X)$, and $h^{k,(p,q)}$ be its dimension.
We say that the mixed Hodge structure is of mixed Tate type if $h^{k,(p,q)}=0$ unless $p=q$. We say the mixed Hodge structure on $H^*(X)$ satisfies the curious hard Lefschetz property if there exists a class $\gamma\in W_4H^{2}(X)$ such that for any $i,k\ge0$ the cupping map satisfies
\[
\gamma^{i}:\textup{Gr}_{2d-2i}^WH^{k}(X)\xrightarrow{\cong} \textup{Gr}_{2d+2i}^WH^{k+2i}(X),
\]
where $d=\dim_\BC X$. We say that the mixed Hodge structure of $H^*(X)$ satisfies the numerical curious hard Lefschetz property if the identities above are equalities for dimensions only.

The mixed Hodge structure on the cohomology groups of cluster varieties are studied in \cite{LS,LS1}. They proved that even-dimensional full rank cluster varieties of Louise type satisfy the curious hard Lefschetz property, give an explicit formula for top weight piece in the mixed Hodge structure and proved various vanishing of cohomology as well as computed mixed Hodge numbers $h^{k,(p,p)}$ when $p$ is small for acyclic cluster varieties.

\subsection{Main Results.}
In this paper, we focus on the mixed Hodge structures of 3-dimensional cluster varieties without further hypothesis.  We show that 3-dimensional cluster varieties are always of Louise type (Proposition \ref{finite}), so that the Louise algorithm can always be applied. We give the explicit formula of mixed Hodge numbers of 3-dimensional cluster varieties of full rank as follows.

\begin{thm}[Proposition \ref{1m},\ref{2m}] \label{1.1}
    Let $\CA=\CA(\widetilde{B})$ be the 3-dimensional cluster variety of associated with a full rank matrix $\widetilde{B}$. Then the mixed Hodge numbers $h^{k,(p,p)}(\CA)$ are
    \begin{center}
\begin{tabular}{c|cccc}
$k-p$ & $H^0$& $H^1$& $H^2$&$H^3$ \\
\hline
$0$ & $1$ &$m$&$m$&$1$\\
$1$& $0$ &$0$&$C$& $C$ 
\end{tabular}
\end{center}
where $m=1,2,3$ is the number of frozen variables and $C$ is a combinatorial constant determined by the matrix $\widetilde{B}$. 
\end{thm}

The strategy is to inductively apply Mayer-Vietoris sequence to the covering induced by the freezing operation of the separating edge in the Louise algorithm, following \cite{LS}. In the induction process, we use de Rham cohomology theory to choose good differential forms to represent a basis of Deligne splitting $H^{k,(p,p)}$ on each open set and glue them to give a global form using Mayer-Vietoris sequences.

The non-full rank case is far more complicated for the following two reasons. On one hand, the cluster varieties of non-full rank could have singularities, so de Rham cohomology fails and it is not known how to compute using the Mayer-Vietoris sequence even if we still have the open cover. On the other hand, the mixed Hodge structures of the induction basis, called isolated cluster varieties, are unknown in non-full rank case. Nevertheless, many cluster varieties of non-full rank are still smooth, especially when the directed graph is an acyclic triangle, we have:

\begin{thm}[Proposition \ref{3m}] \label{1.2}
Let $\CA$ be the cluster variety associated with 
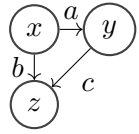
\begin{center}
\begin{tikzpicture}
\tikzstyle{round}=[circle,thick,draw=black!75, minimum size=4mm];
\tikzstyle{square}=[rectangle, thick,draw=black!75, minimum size=6mm];
\begin{scope}
\node[round] (1) {$x$};
\node[round](2)[right of =1]{$y$};
\node[round](3)[below of =1]{$z$};
\draw (1) edge[->,"$a$"] (2);
\draw (1) edge[->,"$b$",swap] (3);
\draw (2) edge[->,"$c$"] (3);
\end{scope}    
\end{tikzpicture}
\end{center}
where $a,b,c$ are positive integers. Then the mixed Hodge structure of $\CA$ is of mixed Tate type and the mixed Hodge numbers $h^{k,(p,p)}(\CA)$ are
\begin{center}
\begin{tabular}{c|cccc}
$k-p$ & $H^0$& $H^1$& $H^2$&$H^3$ \\
\hline
$0$ & $1$ & $0$ & $1$ & $1$\\
$1$ & $0$ &$0$& $C$& $C+1$ \\
\end{tabular}
\end{center}
where $C= \gcd(a, b) +\gcd(a,c) +\gcd(b,c) - 3$. In particular, the mixed Hodge structure is of mixed Tate type but does not satisfy the numerical curious hard Lefschetz property.  
\end{thm}

There are two types of 3-dimensional cluster varieties not covered by Theorem \ref{1.1} and \ref{1.2}. They are defined by the following directed graphs.
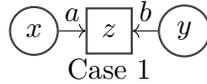
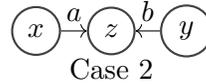
\begin{center}
\begin{tikzpicture}
\tikzstyle{round}=[circle,thick,draw=black!75, minimum size=4mm];
\tikzstyle{square}=[rectangle, thick,draw=black!75, minimum size=6mm];
\begin{scope}
\node[square] (1) {$z$};
\node[round] (2) [left of = 1]{$x$};
\node[round] (3) [right of = 1]{$y$};
\draw (2) edge[->,"$a$"] (1);
\draw (3) edge[->,"$b$",swap] (1);
\node at (0,-0.5) {Case 1};
\node[round] (1)  at (5,0) {$z$};
\node[round] (2) [left of = 1]{$x$};
\node[round] (3) [right of = 1]{$y$};
\draw (2) edge[->,"$a$"] (1);
\draw (3) edge[->,"$b$",swap] (1);
\node at (5,-0.5) {Case 2};
\end{scope}
\end{tikzpicture}
\end{center}
Case 2 can easily be reduced to Case 1 using Mayer-Vietoris sequence and Louise algorithm. However, Case 1 involves singular cluster varieties and is complicated. We illustrate the calculation of the mixed Hodge numbers when $a=b=1$.

\begin{thm} [Theorem \ref{h},\ref{ih}]
   Let $\CA$ be the singular cluster variety associated with the directed diagram
   
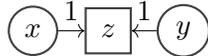
\begin{center}
\begin{tikzpicture}
\tikzstyle{round}=[circle,thick,draw=black!75, minimum size=4mm];
\tikzstyle{square}=[rectangle, thick,draw=black!75, minimum size=6mm];
\begin{scope}
\node[square] (1) {$z$};
\node[round] (2) [left of = 1]{$x$};
\node[round] (3) [right of = 1]{$y$};
\draw (2) edge[->,"$1$"] (1);
\draw (3) edge[->,"$1$",swap] (1);
\end{scope}
\end{tikzpicture}
\end{center}
Then its mixed Hodge numbers of cohomology $h^{k,(p,p)}(\CA)$ and intersection cohomology $ih^{k,(p,p)}(\CA)$ are
\begin{center}
\begin{tabular}{c|cccc}
$k-p$ & $H^0$& $H^1$& $H^2$&$H^3$ \\
\hline
$0$ & $1$ &$1$&$2$&$1$\\
  &     &   & &   
\end{tabular}
~~~~~~\begin{tabular}{c|cccc}
$k-p$ & $IH^0$& $IH^1$& $IH^2$&$IH^3$ \\
\hline
$0$ & $1$ &$1$&$2$&$1$\\
$1$ &     &   &$1$&   
\end{tabular}
\end{center}
In particular, the mixed Hodge structures are of mixed Tate type but do not satisfy the numerical hard Lefschetz property.
\end{thm}

We use classical mixed Hodge-theoretic approaches such as compactification and resolution of singularities together with weight spectral sequence to calculate the mixed Hodge structures. We anticipate that a similar method will apply to the case when the weights in the graph other than 1.

The motivation of finding the mixed Hodge numbers of cluster varieties is to understand the $P=W$ phenomenon for cluster varieties. In \cite{Z,Z1}, a $P=W$ phenomenon is conjectured for cluster varieties as a real analytic counterpart of the $P=W$ conjecture for Higgs bundles proposed in \cite{dCHM} and proved in \cite{MS,HMMS,MSY} independently. The $P=W$ phenomenon for cluster varieties predicts that for any cluster variety $\CA$, there is a proper real analytic torus fibration $h:\CA\to \BR^{\dim_\BC \CA}$ such that the mixed Hodge-theoretic weight filtration of $\CA$ is identified with the perverse filtration associated with $h$. In \cite{Z}, 2-dimensional cases are proved. In \cite{Z1}, the conjectural morphisms $h$ are constructed for isolated cluster varieties, and the $P=W$ phenomenon is proved for isolated cluster varieties of full rank. For non-isolated case, there are no known systematic construction of the morphism $h$ so far. We anticipate that the of mixed Hodge numbers of 3-dimensional cluster varieties may shed some lights on constructing the morphism $h$ by providing conjectural perverse numbers, monodromy groups and topological properties of the fibers of the fibers of $h$. 

\subsection{Outline} This paper is organized as follows. 

In Section 2, we review the definition of cluster varieties and the Louise algorithm. We also discuss the mixed Hodge structures of 2-dimensional cluster varieties.

In Section 3, we study the mixed Hodge structure of smooth 3-dimensional clusters by applying the Mayer-Vietoris sequences to the separating edges in the Louise algorithm. We calculate the mixed Hodge numbers and show that the numerical curious Lefschetz property fails if the cluster variety is not of full rank. We also present a basis using algebraic differential forms for the Deligne splitting when the cluster varieties are of full rank. When the graph does not satisfy the Louise property, we also show that the corresponding cluster algebra is not of finite type.

In Section 4, we study a specific singular cluster variety and compute its mixed Hodge numbers by constructing an explicit logarithmic resolution and apply the weight spectral sequence. We also compute the mixed Hodge structure of its intersection cohomology groups. 

\subsection{Acknowledgements.} We thank Mark de Cataldo, Lie Fu, Thomas Lam, Wanchun Shen, David Speyer and Xiping Zhang for helpful conversations. The second author is partially supported by the Fundamental Research Funds for the Central Universities.

\section{Preliminaries} 

\subsection{Cluster varieties} \label{2}

An \textit{extended exchange matrix} $\widetilde{B} = (\widetilde{B}_{ij})$ is a $(n + m) \times n$ matrix with integer coefficients, whose upper $n \times n$ submatrix, the \textit{principal part} $B$, is skew-symmetric. A \textit{seed} $t = (x, \widetilde{B})$ consists of $n + m$ algebraically independent variables $x_1, x_2, \dots, x_{n+m}$ and an extended exchange matrix $\widetilde{B}$. The variables $x_1, \dots, x_n$ are called \textit{mutable variables} and $x_{n+1}, \dots, x_{n+m}$ are called \textit{frozen variables}. We sometimes write $y_1,\dots,y_m$ for $x_{n+1},\dots,x_{n+m}$.

The \textit{mutation of a seed} $t = (x, \widetilde{B})$ \emph{in the direction} $k \in \{1, 2, \dots, n\}$ generates a new seed $t' = \mu_k(t)= (x', \widetilde{B}') $, following the rules:

\begin{enumerate}
\item The mutated matrix $\widetilde{B}' = \mu_k(\widetilde{B})$ is defined by:
\begin{equation}  \label{mutation1}
\widetilde{B}'_{ij} = 
\begin{cases} 
    -\widetilde{B}_{ij}, & \text{if } i = k \text{ or } j = k, \\
    \widetilde{B}_{ij} + [\widetilde{B}_{ik}]_+ [\widetilde{B}_{kj}]_+ - [\widetilde{B}_{ik}]_- [\widetilde{B}_{kj}]_-, & \text{otherwise},
\end{cases}
\end{equation}
where $[x]_+ = \max(x, 0)$ and $[x]_- = \min(x, 0)$.

\item The mutated variables $x'_i$ are given by:
\begin{equation} \label{mutation2}
x'_i = 
\begin{cases} 
    \frac{\prod x_j^{[\widetilde{B}_{jk}]_+} + \prod x_j^{-[\widetilde{B}_{jk}]_-}}{x_k}, & \text{if } i = k, \\
    x_i, & \text{if } i \neq k.
\end{cases}
\end{equation}
\end{enumerate}
By repeatedly applying mutations in all possible directions, we generate a collection of seeds. Mutable variables change across the seeds, while frozen variables remain invariant. The \emph{cluster algebra} defined by the initial seed $t$, denoted as $A(x, \widetilde{B})$ or simply $A(\widetilde{B})$ or $A$, is defined as the $\mathbb{C}$-subalgebra of the rational function field $\mathbb{C}(x_1, x_2, \dots, x_{n+m})$ generated by all mutable variables during the mutations and the inverses of the frozen variables. The \emph{cluster variety} is the affine scheme:
\[
\mathcal{A}(x, \widetilde{B}) = \text{Spec}\, A(x, \widetilde{B}).
\]

We remark that the cluster algebra are not always of finite type, and hence when we talk about cluster varieties we always have to check that the cluster algebra is of finite type. The Louise property is introduced to ensure finite generation. To illustrate the Louise property, we need the following equivalent definition of cluster varieties via directed graphs. 

Given a seed $t = (x, \widetilde{B})$, the associated directed graph $\Gamma(t) = \Gamma(\widetilde{B})$ has vertices $\{1, \dots, n+m\}$, where there is a directed edge of weight $\widetilde{B}_{ij}$ from $i$ to  $j$ if $\widetilde{B}_{ij} > 0$. The vertices $1,\dots, n$ representing mutable variables are called mutable vertices, and are denoted by circles in the graph. The remaining vertices are frozen vertices, denoted by squares in the graph. The subgraph of mutable vertices is called the mutable subgraph, denoted as $\Gamma(t)_{mut}$. When $\Gamma(t)_{mut}$ has no edges, the corresponding cluster variety is called \textit{isolated}. 

To freeze a subset of mutable variables $S\subset\{1,\dots,n\}$ is simply to regard all $x_i$ as frozen variables for $i\in S$. We denote $A_S$ the cluster algebra after freezing the cluster variables $\{x_i \mid i \in S\}$. In this situation, $A_S \cong A[x_i^{-1}\mid i\in S]$. A direct edge $e:i \to j$ in $\Gamma(t)_{mut}$  is called a \textit{separating edge} if there is no bi-infinite path through $e$.  \cite[Proposition 2.1]{Muller2013} shows that if $i \to j$ is a separating edge, then $x_i$ and $x_j$ are not simultaneously zero on the cluster variety $\mathcal{A}$, \emph{i.e.} $\mathcal{A}$ is covered by the two open subsets $\text{Spec } \mathcal{A}[x_i^{-1}]$ and $\text{Spec } \mathcal{A}[x_j^{-1}]$.

A cluster algebra $A$ satisfies the \textit{Louise property} if it satisfies:
\begin{enumerate}
    \item For some seed $t$ of $A$, the quiver $\Gamma(t)_{mut}$ has no edges, or
    \item For some seed $t = (x, \widetilde{B})$ of $A$, the quiver $\Gamma(t)_{mut}$ has a separating edge $i \to j$, such that the cluster algebras $A_{\{i\}}$, $A_{\{j\}}$, and $A_{\{i, j\}}$ all satisfy the Louise property.
\end{enumerate}

The recursive operations of freezing vertices of separating edges are called the Louise algorithm. When the cluster algebra $A$ satisfy the Louise property, we say that the corresponding variety is of Louise type. It follows from the definition that if $\mathcal{A}$ of Louise type, then all cluster varieties arising in the recursive definition also of Louise type. In particular, as established in \cite{Muller2013}, cluster varieties defined by acyclic graphs are of Louise type.

Let $\CA$ be a cluster variety of Louise type. If $\CA$ is isolated with $n$ mutable and $m$ frozen variables, then the mutations of $\CA$ in different directions are commutative and hence $\CA$ is of dimensional $n+m$ with $n$ defining equations
\[
\begin{cases}
x_1x_1'&=\prod_j y_j^{[\widetilde{B}_{j,1}]_+}+\prod_j y_j^{[\widetilde{B}_{j,1}]_-},\\
&\cdots\\
x_nx_n'&=\prod_j y_j^{[\widetilde{B}_{j,n}]_+}+\prod_j y_j^{[\widetilde{B}_{j,n}]_-}\\
\end{cases}
\]
in $\BC^{2n}\times(\BC^*)^m=\Spec \BC[x_1,\dots, x_n,y_1^{\pm1},\dots,y_m^{\pm1}]$.
 
If $\CA$ is not isolated, then by definition there exists a separating edge $i\to j$ such that $\CA=\CA_{\{i\}}\cup\CA_{\{j\}}$ and $\CA_{\{i\}}\cap\CA_{\{j\}}=\CA_{\{i,j\}}$. Since freezing variables inverts the corresponding variables in the cluster algebra, the cluster variety obtained from freezing operation is an open subvariety. Therefore, an inductive argument implies that the dimension of cluster variety $\CA$ of Louise type is always $n+m$ and iterated Mayer-Vietoris arguments applied to the covering $\CA=\CA_{\{i\}}\cup\CA_{\{j\}}$ in the Louise algorithm calculate the Hodge structure of $\CA$.

\subsection{Deligne Splitting for Mixed Hodge Structures} \label{tate}
Let $(V,W,F)$ be a mixed Hodge structure. Then the graded piece $(\textup{Gr}_{p+q}^WV,F)$ is a pure Hodge structure of weight $p+q$. When $V=H^k(X)$ is the cohomology group of an algebraic variety $X$ over $\BC$, we denote by $H^{k,(p,q)}$ its $(p,q)$ component $(\textup{Gr}_{p+q}H^k(X))^{pq}$. The decomposition of \(H^k(X, \mathbb{C})\) into subspaces \(H^{k,(p,q)}\) is known as the Deligne splitting. Their dimensions \(h^{k,(p,q)} := \dim H^{k,(p,q)}\) are called the mixed Hodge numbers of $X$. We say that the mixed Hodge structure $H^*(X)$ is of \textit{mixed Tate type} if $H^{k,(p,q)} = 0 $ for $p \neq q.$ The following result follows formally from the fact that the Deligne splitting is functorial and is respected by the boundary maps in Mayer–Vietoris sequences:

\begin{lem} \label{MVH}
Suppose that a complex algebraic variety $X$ is covered by two open subvarieties \(U\) and \(V\). If the cohomology groups of \(U\), \(V\), and \(U \cap V\) are of mixed Tate type, then so is \(X\). In this case, denote
\[
f^{k,(p,p)}: H^{k,(p,p)}(U)\oplus H^{k,(p,p)}(V) \to H^{k,(p,p)}(U\cap V),
\]
then there is a non-canonical isomorphism of complex vector spaces
\[
H^{k,(p,p)}(X)=\ker f^{k,(p,p)}\oplus \textup{coker} f^{k-1,(p,p)}.
\]
\end{lem}

We say the mixed Hodge structure on $H^*(X)$ satisfies the \emph{curious hard Lefschetz property} if there exists a class $\gamma\in W_4H^{2}(X)$ such that for any $i,k\ge0$ the cupping map satisfies
\[
\gamma^{i}:\textup{Gr}_{2d-2i}^WH^{k}(X)\xrightarrow{\cong} \textup{Gr}_{2d+2i}^WH^{k+2i}(X),
\]
where $d=\dim_\BC X$. We say that the mixed Hodge structure of $H^*(X)$ satisfies the \textup{numerical curious hard Lefschetz property} if the identities above are equalities for dimensions only.

\subsection{Two conventions} \label{notation}
Since cluster varieties are partial compactification of complex tori $(\BC^*)^k$, the language of logarithmic forms are very useful to describe the cohomology groups of cluster varieties. To simplify notations, we introduce two conventions when using log forms in this section.

Let $(\BC^*)^k$ be a torus with coordinates $x_1,\cdots,x_k$. Then the regular differential forms on $(\BC^*)^k$ are of the type
\[
\omega=\sum_{I=\{i_1,\cdots,i_d\}\subset \{1,\cdots k\}} f_Id\log x_{i_1}\wedge\cdots\wedge d\log x_{i_d}, ~~ f_I\in \BC[x_1^{\pm1},\cdots,x_k^{\pm 1}].
\]
We remark that when all functions $f_I$ are constant, they are nothing but the log forms defined in \cite{LS}, and log $p$-forms are always of Hodge type $H^{p,(p,p)}$. Since such forms will appear repeatedly in the study of cluster varieties, we introduce the following convention to simplify notations.

\medskip 

\noindent
{\bf Convention.}  We will  write $\underline{x}$ for $d\log x$. For higher differential forms we simply drop wedge symbols, \emph{e.g.} $\underline{x}\wedge \underline{y}$ is simplified as $\underline{xy}$.  

In applications, the variables appear in the log forms are the cluster variables in the initial seed, so there won't be ambiguities.
\medskip

Let $X$ be a smooth cluster variety of Louise type, and $U\cong (\BC^*)^{n+m}$ be one its cluster tori. Then $U$ is open dense in $X$. Although it is not guaranteed in general that a differential form on $U$ can be extended to $X$, but if such extension exists, it is unique. Furthermore, if a differential form on $U$ is closed and extends to $X$, then the extension is also closed. Therefore, to represent a cohomology class on $X$ by a differential form, it suffices to specify a closed differential form on $U$ which is extendable to $X$. To simplify notation, we will use the following convention.

\medskip
\noindent

\noindent{\bf Convention.} Let $X$ be a smooth variety and $U\subset X$ be an open dense subset. Let $\alpha\in H^k(X)$ be a cohomology class and let $\omega\in\Gamma(U,\Omega^k)$ be a differential form on $U$. We will simply say $\alpha$ is represented by $\omega$ if $\omega$ admits a regular extension to $X$ which represents $\alpha$.  \\

\noindent

\noindent{\bf Warning.} In the notations in the Convention above, the statement is \emph{not} equivalent to $\alpha|_U$ is represented by $\omega$. In fact, it is possible that $\omega$ is exact on $U$ but its unique extension to $X$ is not exact.

\subsection{Mixed Hodge structures of 2D cluster varieties}
Cluster varieties of dimension 2 are studied in \cite[Section 7]{LS} and \cite{Z}. It is straightforward from the definition that cluster varieties associated with a weighted directed graph with 2 vertices are always of Louise type. We recall the following result.
\begin{thm} \label{2d}
Let $\CA$ be a cluster variety with 1 mutable and 1 frozen variable. Then the corresponding graph $\Gamma$ is of the form 
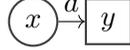
\begin{center}
\begin{tikzpicture}
\tikzstyle{round}=[circle,thick,draw=black!75, minimum size=4mm];
\tikzstyle{square}=[rectangle, thick,draw=black!75, minimum size=6mm];

\begin{scope}
\node[round] (1) {$x$};
\node[square] (2) [right of = 1]{$y$};
\draw (1) edge[->,"$a$"] (2);
\end{scope}
\end{tikzpicture}
\end{center}
and the cluster variety is defined by the equation
\[
\CA=\{(x,y,z)\in\BC^3\mid xx'=y^a+1,~~y\neq0\}.
\]
The mixed Hodge structure on the cohomology of $X$ is of mixed Tate type, and the mixed Hodge numbers $\dim H^{k,(p,p)}$ are
\begin{center}
\begin{tabular}{c|ccc}
$k-p$ & $H^0$& $H^1$& $H^2$ \\
\hline
$0$ & $1$ &$1$&$1$\\
$1$ & &  & $a-1$
\end{tabular}
\end{center}
Furthermore, following the Conventions in Section \ref{notation}, the Deligne splitting admits a basis which are represented by the following differential forms
\begin{center}
\begin{tabular}{c|ccc}
$k-p$ & $H^0$& $H^1$& $H^2$ \\
\hline
$0$ & $\langle1\rangle$ &$\langle\underline{y}\rangle$&$\langle\underline{xy}\rangle$\\
$1$ & $0$& $0$ & $\langle y^i\underline{xy}\mid 1\le i\le a-1\rangle$
\end{tabular}
\end{center}
\end{thm}

We remark that the differential forms $y^i\underline{xy}$ ($0\le i\le a-1$) which are defined a priori on $U=\CA\setminus\{x=0\}\cong\BC^*\times\BC^*$ do extend to $\CA$; see \cite[Section 7.1]{LS}.

\section{Smooth cluster varieties in dimension 3}\label{smooth}
\setcounter{subsection}{-1}
In this section, we present an explicit calculation of the mixed Hodge numbers of smooth 3-dimensional cluster varieties. We will show that all the cohomology classes are of mixed Tate type in the sense of Section \ref{tate} and present them in terms of algebraic differential forms. We will simply use $\CA$ for cluster varieties without the reference to the matrix when no confusion arises. Since the mixed Hodge numbers $\dim H^{k,(p,p)}(\CA)$ are mainly concentrated where $p$ is slightly less than or equal to $k$, they will be listed in tables indexed by $k$ and $k-p$.

\subsection{0 mutable vertices}
If all vertices are frozen, the cluster variety is isomorphic to $(\BC^*)^3$. Since $\BC^*$ has two nontrivial cohomology classes with $\dim H^{0,(0,0)}(\BC^*)=\dim H^{1,(1,1)}(\BC^*)=1$, K\"unneth formula implies that $\dim H^{k,(p,p)}(X)$ is
\begin{center}
\begin{tabular}{c|cccc}
$k-p$ & $H^0$& $H^1$& $H^2$&$H^3$ \\
\hline
0 & 1 &3&3&1
\end{tabular}
\end{center}
Following the convention in Section \ref{notation}, the Deligne splitting admits a basis represented by  
\[\begin{cases}
H^{0,(0,0)}(\CA)=\langle 1\rangle\\
H^{1,(1,1)}(\CA)=\langle \underline{x},\underline{y} , \underline{z} \rangle\\
H^{2,(2,2)}(\CA)=\langle \underline{xy}, \underline{yz}, \underline{xz}\rangle\\
H^{3,(3,3)}(\CA)=\langle \underline{x}\underline{y}\underline{z}\rangle.  
\end{cases}
\]

\subsection{1 mutable vertex}
Suppose that the cluster variety $X$ is defined by a graph with 1 mutable and 2 frozen vertices. Then the extended exchange matrix is of the form
\[
\left(\begin{array}{c}
0\\
a\\
b
\end{array}
\right)
\]
where $a$ and $b$ are not simultaneously 0. We first treat the case when one of $a,b$ is 0, say $b=0$, then the graph is
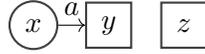
\begin{center}
\begin{tikzpicture}
\tikzstyle{round}=[circle,thick,draw=black!75, minimum size=4mm];
\tikzstyle{square}=[rectangle, thick,draw=black!75, minimum size=6mm];

\begin{scope}
\node[round] (1) {$x$};
\node[square] (2) [right of = 1]{$y$};
\node[square] (3) [right of = 2]{$z$};
\draw (1) edge[->,"$a$"] (2);
\end{scope}
\end{tikzpicture}
\end{center}
 and by definition
\[
\CA\left(\begin{array}{c}
0\\
a\\
0
\end{array}\right)=\CA\binom{0}{a}\times \BC^*=\{(x,x',y)\mid xx'=y^a+1,y\ne0\}\times \{z\mid z\ne 0\}
\]
is just a product. By Proposition \ref{2d} and K\"unneth formula,  we have

\begin{prop} \label{10}
Let $\CA$ be the cluster variety associated with the graph
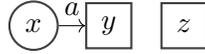
\begin{center}
\begin{tikzpicture}
\tikzstyle{round}=[circle,thick,draw=black!75, minimum size=4mm];
\tikzstyle{square}=[rectangle, thick,draw=black!75, minimum size=6mm];

\begin{scope}
\node[round] (1) {$x$};
\node[square] (2) [right of = 1]{$y$};
\node[square] (3) [right of = 2]{$z$};
\draw (1) edge[->,"$a$"] (2);
\end{scope}
\end{tikzpicture}
\end{center}
where $a\neq 0$. Then the mixed Hodge numbers $h^{k,(p,p)}(\CA)$ are
\begin{center}
\begin{tabular}{c|cccc}
$k-p$ & $H^0$& $H^1$& $H^2$&$H^3$ \\
\hline
$0$ & $1$ &$2$&$2$&$1$\\
$1$ & $0$& $0$ & $a-1$ & $a-1$
\end{tabular}
\end{center}
Following the convention in Section \ref{notation}, the Deligne splitting $H^{k,(p,p)}$ admits a basis represented by 
\begin{equation} \label{11}
\begin{cases}
H^{0,(0,0)}=\langle1\rangle\\
H^{1,(1,1)}=\langle \underline{y}, \underline{z}\rangle\\
H^{2,(2,2)}=\langle \underline{xy},\underline{yz}\rangle\\
H^{3,(3,3)}=\langle \underline{xyz}\rangle\\
H^{2,(1,1)}= \langle y^i\underline{xy}\mid 1\le i\le a-1\rangle\\
H^{3,(2,2)}= \langle y^i\underline{xyz}\mid 1\le i\le a-1\rangle.\\
\end{cases}
\end{equation}
\end{prop} 

\medskip

Suppose that $a,b$ are both nonzero. Then the cluster varieties with 1 mutable variables are defined by directed graphs of the form

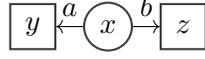
\begin{center}
\begin{tikzpicture}
\tikzstyle{round}=[circle,thick,draw=black!75, minimum size=4mm];
\tikzstyle{square}=[rectangle, thick,draw=black!75, minimum size=6mm];

\begin{scope}
\node[round] (1) {$x$};
\node[square] (2) [left of = 1]{$y$};
\node[square] (3) [right of = 1]{$z$};
\draw (1) edge[->,"$a$",swap] (2);
\draw (1) edge[->,"$b$"] (3);
\end{scope}
\end{tikzpicture}
\end{center}
where the arrows may reverse depending on the signs of $a$ and $b$. We have the following observation.

\begin{lem} \label{isom}
Suppose $a,b>0$, $g=\gcd(a,b)$. Then the cluster variety 
\[
\CA=\{(x,x',y,z)\mid xx'=y^az^b+1,~~y,z\neq0\}
\]
is isomorphic to 
\[
\widetilde{\CA}=\{(u,u',v)\mid uu'=v^g+1,~~v\neq0\}\times \{w\mid w\neq0\}.
\]
More explicitly, if we choose $s,t$ such that $sa+tb=g$ and denote $a=a'g$ and $b=b'g$, then the map $\widetilde{\CA}\to \CA$ defined by 
\begin{equation} \label{change}
x=u,~~ x'=u', ~~y=v^sw^{b'},~~ z=v^tw^{-a'}
\end{equation}
is an isomorphism.
\end{lem} 
\begin{proof}
It is straightforward to check that 
\[
x=u,~~x'=u',~~v=y^{a'}z^{b'},~~w=y^tz^{-s}
\]
gives the inverse.
\end{proof}

\begin{prop} \label{1m}
Let $\CA$ be the cluster variety associated with the diagram
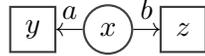
\begin{center}
\begin{tikzpicture}
\tikzstyle{round}=[circle,thick,draw=black!75, minimum size=4mm];
\tikzstyle{square}=[rectangle, thick,draw=black!75, minimum size=6mm];

\begin{scope}
\node[round] (1) {$x$};
\node[square] (2) [left of = 1]{$y$};
\node[square] (3) [right of = 1]{$z$};
\draw (1) edge[->,"$a$",swap](2);
\draw (1) edge[->,"$b$"](3);
\end{scope}
\end{tikzpicture}
\end{center}
where $a,b$ are not simultaneously 0. Then $\CA$ is of mixed Hodge-Tate type and its mixed Hodge numbers $h^{k,(p,p)}$ are
\begin{center}
\begin{tabular}{c|cccc}
$k-p$ & $H^0$& $H^1$& $H^2$&$H^3$ \\
\hline
$0$ & $1$ &$2$&$2$&$1$\\
$1$ & $0$& $0$ & $\gcd(a,b)-1$ & $\gcd(a,b)-1$
\end{tabular}
\end{center}
Furthermore, following the convention in Section \ref{notation}, the Deligne splitting admits a basis represented by
\begin{equation} \label{12}
\begin{cases}
H^{0,(0,0)}=\langle1\rangle\\
H^{1,(1,1)}=\langle \underline{y}, \underline{z}\rangle\\
H^{2,(2,2)}=\langle a\underline{xy}+b\underline{xz},\underline{yz}\rangle\\
H^{3,(3,3)}=\langle \underline{xyz}\rangle\\
H^{2,(1,1)}= \langle y^{ia/(a,b)}z^{ib/(a,b)}(a\underline{xy}+b\underline{xz})\mid 1\le i\le (a,b)-1\rangle\\
H^{3,(2,2)}= \langle y^{ia/(a,b)}z^{ib/(a,b)}\underline{xyz}\mid 1\le i\le (a,b)-1\rangle.\\
\end{cases}
\end{equation}

\end{prop}

\begin{proof}
The case where $a=0$ or $b=0$ are treated in Proposition \ref{10}. For general case, we first study the mixed Hodge structure of the variety  
\begin{equation} \label{0101}
\CA'=\{(x,x',y,z)\mid xx'=1+y^az^b,\,y,z\neq0\}.
\end{equation}
By Lemma \ref{isom}, we have 
\[
\CA'\cong\{(u,u',v)\mid uu'=v^g+1,~~v\neq0\}\times \{w\mid w\neq0\},
\]
where $g=(a,b)$. So $\CA'$ is isomorphic to the cluster variety associated with the graph
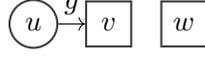
\begin{center}
\begin{tikzpicture}
\tikzstyle{round}=[circle,thick,draw=black!75, minimum size=4mm];
\tikzstyle{square}=[rectangle, thick,draw=black!75, minimum size=6mm];

\begin{scope}
\node[round] (1) {$u$};
\node[square] (2) [right of = 1]{$v$};
\node[square] (3) [right of = 2]{$w$};
\draw (1) edge[->,"$g$"] (2);
\end{scope}
\end{tikzpicture}
\end{center}

Therefore, the mixed Hodge structure of $\CA'$ and follows from Proposition \ref{10}, presented in variables $u,v,w$. So it suffices to apply the change of variables (\ref{change}) to convert into variables $x,y,z$. By (\ref{change}) we have
\begin{equation}\label{410}
\underline{v}=d\log v=d\log (y^{a'}z^{b'})=a'd\log y+b'd\log z=a'\underline{y}+b'\underline{z}
\end{equation}
and
\begin{equation} \label{411}
\underline{w}=d\log w=d\log (y^tz^{-s})=td\log y-sd\log z=t\underline{y}-s\underline{z}.
\end{equation}

Plugging equations (\ref{410}) and (\ref{411}) into (\ref{11}), we have:
\begin{enumerate}
\item $k=p=0$. We have $H^{0,(0,0)}(\CA')=\langle1\rangle$.
\item $k=p=1$: We have $\underline{v}$ and  $\underline{w}$ form a basis of $H^{1,(1,1)}(\CA')$.  Since the matrix $\left(\begin{array}{cc} a'&b'\\t&-s\end{array}\right)$ has determinant -1, equations (\ref{410}) and (\ref{411}) imply that $H^{1,(1,1)}(\CA')$  is also spanned by $\underline{y}$ and $\underline{z}$.
\item $k=p=2$: By a simple change of variables, we have
\[
\underline{uv}=\underline{x}(a'\underline{y}+b'\underline{z})=a'\underline{xy}+b'\underline{xz}
\]
and
\[
\underline{vw}=(a'\underline{y}+b'\underline{z})(t\underline{y}-s\underline{z})=-\underline{yz}.
\]
So
\[
H^{2,(2,2)}(\CA')=\langle a\underline{xy}+b\underline{xz},\underline{yz}\rangle.
\]
\item $k=p=3$: Similarly 
\[
\underline{uvw}=\underline{x}(a'\underline{y}+b'\underline{z})(t\underline{y}-s\underline{z})=-\underline{xyz},
\] 
so we have
\[
H^{3,(3,3)}(\CA')=\langle \underline{xyz}\rangle.
\]
\item $k=2,p=1$: Basis are 
\[
v^i\underline{uv}=(y^{a'}z^{b'})^i\underline{x}(a'\underline{y}+b'\underline{z})=y^{i a'}z^{i b'}(a\underline{xy}+b\underline{xz}),
\]
where $1\le i\le \gcd(a,b)-1$.
\item $k=3,p=2$: Basis are $y^{i a'}z^{i b'}\underline{xyz}$, where $1\le i\le \gcd(a,b)-1$.
\end{enumerate}
Therefore, the list of basis of Deligne splitting (\ref{12}) holds for $\CA'$.

\medskip

Back to the calculation of cluster varieties. Recall that $\CA$ is defined by 
\[
xx'=1+y^{|a|}z^{|b|}, ~~y,z\ne 0 ~~ \textup{if }ab>0
\]
or
\[
xx'=y^{|a|}+z^{|b|}, ~~y,z\ne 0,~~ \textup{if } ab<0.
\]
Following the proof of \cite[Proposition 3.1]{Z}, after replacing $x'$ by a suitable power of $y$ or $z$, $\CA$ is always isomorphic to $\CA'$. Since in this isomorphism the variables $x,y,z$ are fixed, so the list of basis of Deligne splittings of $\CA'$ also holds for $\CA$.
\end{proof}

\subsection{2 mutable vertices}
In this section, we study the mixed Hodge structures of cluster varieties associated with a graph with 2 mutable vertices and 1 frozen variable. The extended exchange matrices are of the form
\[
\left(
\begin{array}{cc}
0&-a\\
a&0\\
b&c\\
\end{array}
\right).
\]
So the corresponding directed graph is of the form

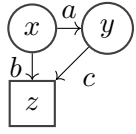
\begin{center}
\begin{tikzpicture}
\tikzstyle{round}=[circle,thick,draw=black!75, minimum size=4mm];
\tikzstyle{square}=[rectangle, thick,draw=black!75, minimum size=6mm];
\begin{scope}
\node[round] (1) {$x$};
\node[round](2)[right of =1]{$y$};
\node[square](3)[below of =1]{$z$};
\draw (1) edge[->,"$a$"] (2);
\draw (1) edge[->,"$b$",swap] (3);
\draw (2) edge[->,"$c$"] (3);
\end{scope}    
\end{tikzpicture}
\end{center}
where the arrow $x\to z$ and $y\to z$ may possibly disappear or reverse depending on whether $b$ and $c$ are zero or negative. When $a=0$, the cluster variety is possibly singular, and will be discussed in Section 3. By symmetry between the variables $x$ and $y$, we may assume $a>0$ in this section.

\begin{prop}\label{2m}
Let $\CA$ be the cluster variety associated with the graph
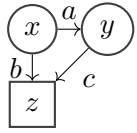
\begin{center}
\begin{tikzpicture}
\tikzstyle{round}=[circle,thick,draw=black!75, minimum size=4mm];
\tikzstyle{square}=[rectangle, thick,draw=black!75, minimum size=6mm];
\begin{scope}
\node[round] (1) {$x$};
\node[round](2)[right of =1]{$y$};
\node[square](3)[below of =1]{$z$};
\draw (1) edge[->,"$a$"] (2);
\draw (1) edge[->,"$b$",swap] (3);
\draw (2) edge[->,"$c$"] (3);
\end{scope}    
\end{tikzpicture}
\end{center}
where integers $a>0$. Then $\CA$ is of mixed Hodge-Tate type and its mixed Hodge numbers $h^{k,(p,p)}$ are
\begin{center}
\begin{tabular}{c|cccc}
$k-p$ & $H^0$& $H^1$& $H^2$&$H^3$ \\
\hline
$0$ & $1$ &$1$&$1$&$1$\\
$1$& $0$ &$0$&$C$& $C$ 
\end{tabular}
\end{center}
where $C=\gcd(a,b)+\gcd(a,c)-2$. Furthermore, following the conventions in Section \ref{notation}, the Deligne splitting $H^{k,(p,p)}$ admits a basis represented by
\begin{equation} \label{21}
\begin{cases}
H^{0,(0,0)}=&\langle1\rangle\\
H^{1,(1,1)}=&\langle \underline{z}\rangle\\
H^{2,(2,2)}=&\langle a\underline{xy}+b\underline{xz}+c\underline{yz}\rangle\\
H^{3,(3,3)}=&\langle \underline{xyz}\rangle\\
H^{2,(1,1)}=&\langle x^{\frac{ia}{(a,c)}}z^{\frac{ic}{(a,c)}}(a\underline{xy}+c\underline{yz}),y^{\frac{jb}{(b,c)}}z^{\frac{jc}{(b,c)}}(a\underline{xy}+b\underline{xz})\\
&\mid 1\le i\le (a,c)-1,1\le j\le(b,c)-1 \rangle\\
H^{3,(2,2)}=& \langle x^{\frac{ia}{(a,c)}}z^{\frac{ic}{(a,c)}}\underline{xyz},y^{\frac{jb}{(b,c)}}z^{\frac{jc}{(b,c)}}\underline{xyz}\\
&\mid 1\le i\le (a,c)-1,1\le j\le (b,c)-1\rangle.\\
\end{cases}
\end{equation}
\end{prop}

\begin{proof}
Since $a\neq0$, the edge $x\rightarrow y$ is a separating edge and we may apply Mayer-Vietoris argument to the covering produced by the Louise algorithm. We have $\CA=U\cup V$, where $U$, $V$ and $U\cap V$ be the cluster varieties associated with the graphs
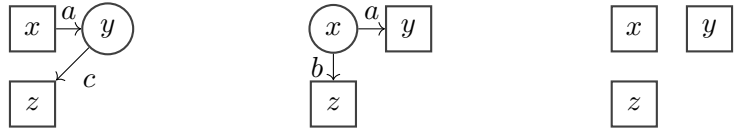
\begin{center}
\begin{tikzpicture}
\tikzstyle{round}=[circle,thick,draw=black!75, minimum size=4mm];
\tikzstyle{square}=[rectangle, thick,draw=black!75, minimum size=6mm];
\begin{scope}
\node[square] (1) {$x$};
\node[round](2)[right of =1]{$y$};
\node[square](3)[below of =1]{$z$};
\draw (1) edge[->,"$a$"] (2);
\draw (2) edge[->,"$c$"] (3);

\node[round] at (4,0) (1) {$x$};
\node[square](2)[right of =1]{$y$};
\node[square](3)[below of =1]{$z$};
\draw (1) edge[->,"$a$"] (2);
\draw (1) edge[->,"$b$",swap] (3);

\node[square] at (8,0)  (1) {$x$};
\node[square](2)[right of =1]{$y$};
\node[square](3)[below of =1]{$z$};
\end{scope}    
\end{tikzpicture}
\end{center}

To calculate the mixed Hodge structure, it suffices to find the kernel and the cokernel of 
\begin{eqnarray*}
f^{k,(p,p)}&:H^{k,(p,p)}(U)\oplus H^{k,(p,p)}(V)&\to H^{k,(p,p)}(U\cap V)\\
                  &(\alpha,\beta)&\mapsto \alpha|_{U\cap V}-\beta|_{U\cap V}.
\end{eqnarray*}
and apply Lemma \ref{MVH}. Note that the Deligne splittings of the cohomology groups of $U$ are $V$ are computed in Proposition \ref{1m}.
\begin{enumerate}
\item $k=p=0$. The cohomology groups $H^{0,(0,0)}(U),H^{0,(0,0)}(V)$ and $H^{0,(0,0)}(U\cap V) $ are 1 dimensional and generated by 1. So $\ker f^{0,(0,0)}=\langle1\rangle$ and $\textup{coker} f^{0,(0,0)}=0$.
\item $k=p=1$. We have 
\[
\begin{cases}
H^{1,(1,1)}(U)=\langle \underline{x},\underline{z}\rangle\\
H^{1,(1,1)}(V)=\langle\underline{y},\underline{z}\rangle\\
H^{1,(1,1)}(U\cap V)=\langle\underline{x},\underline{y},\underline{z}\rangle
\end{cases}
\]
So $\textup{coker} f^{1,(1,1)}=0$ and $\ker f^{1,(1,1)}=\langle  \underline{z}\rangle$.
\item $k=p=2$. We have
\[
\begin{cases}
H^{2,(2,2)}(U)=\langle -a\underline{yx}+c\underline{yz},\underline{xz}\rangle=\langle a\underline{xy}+c\underline{yz},\underline{xz}\rangle\\
H^{2,(2,2)}(V)=\langle a\underline{xy}+b\underline{xz},\underline{yz}\rangle\\
H^{2,(2,2)}(U\cap V)=\langle\underline{xy},\underline{xz},\underline{yz}\rangle
\end{cases}
\]
So $f^{2,(2,2)}$ is a linear map from a 4-dimensional vector space to a 3-dimensional vector space. A simple linear algebra calculation shows that $f^{2,(2,2)}$ surjective and its kernel (viewed as a subspace of $H^2(X)$) is 1-dimensional, and is represented by the extension of $a\underline{xy}+b\underline{xz}+c\underline{yz}$.

\item $k=p=3$. Since the cohomology groups of $H^{3,(3,3)}(U)$ $H^{3,(3,3)}(V)$ $H^{3,(3,3)}(U\cap V)$ are 1 dimensional and generated by $\underline{xyz}$, we have $f^{3,(3,3)}$ is surjective and $\ker f^{3,(3,3)}=\langle\underline{xyz}\rangle$.

\item $k=2,p=1.$ Since $H^{2,(1,1)}(U\cap V)=0$, we have $\textup{coker}\, f^{2,(1,1)}=0$ and  
\[
\begin{split}
\ker f^{2,(1,1)}=H^{2,(1,1)}(U)\oplus H^{2,(1,1)}(V)\\
=\langle x^{i a/\gcd(a,c)}z^{i c/\gcd(a,c)} (a\underline{xy}+c\underline{yz}),\\
y^{j b/\gcd(b,c)}z^{j c/\gcd(b,c)} (a\underline{xy}+b\underline{xz})\rangle,
\end{split}
\]
where $1\le i\le \gcd(a,c)-1$ and $1\le j\le \gcd(b,c)-1$.
\item $k=3,p=2$. Similarly, $H^{3,(2,2)}(U\cap V)=0$, and hence $\textup{coker}\, f^{3,(2,2)}=0$ and 
\[
\begin{split}
\ker f^{3,(2,2)}=\langle x^{ia/\gcd(a,c)}z^{ic/\gcd(a,c)}\underline{xyz},\\
y^{ja/\gcd(a,b)}z^{jb/\gcd(a,b)}\underline{xyz}\rangle,
\end{split}
\]
where $1\le i\le \gcd(a,c)-1$ and $1\le j\le \gcd(b,c)-1$.
\end{enumerate}

Therefore, $f^{k,(p,p)}$ are surjective and $H^{k,(p,p)}(X)=\ker f^{k,(p,p)}$ for all $k,p$. We conclude that the mixed Hodge numbers are
\begin{center}
\begin{tabular}{c|cccc}
$k-p$ & $H^0$& $H^1$& $H^2$&$H^3$ \\
\hline
0 & 1 &1&1&1\\
1& 0 &0&$C$& $C$ 
\end{tabular}
\end{center}
where $C=\gcd(a,b)+\gcd(a,c)-2$.
\end{proof}

\subsection{3 mutable vertices}
Let $\CA$ be a cluster variety which is associated with a graph with 3 mutable variables. We first show that $\CA$ is of Louise type.

\begin{prop} \label{finite}
   Suppose that $A$ is a cluster variety defined by a $3\times3$ matrix which does not satisfy the Louise property, then $A$ is not of finite type. In particular, $\CA=\Spec A$ is not a 3-dimensional variety.
\end{prop}

\begin{proof}
    We first note that for a graph with 3 vertices if the Louise property for $A$ fails, then for any seed, the corresponding graph is an oriented cycle. If not, there is an separating edge, and freezing either or both vertices of it obtains an oriented graph with at most 2 mutable vertices, which are obvious of Louise type. Now suppose the graph of of the form
    
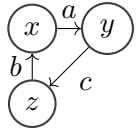
\begin{center}
\begin{tikzpicture}
\tikzstyle{round}=[circle,thick,draw=black!75, minimum size=4mm];
\tikzstyle{square}=[rectangle, thick,draw=black!75, minimum size=6mm];
\begin{scope}
\node[round] (1) {$x$};
\node[round](2)[right of =1]{$y$};
\node[round](3)[below of =1]{$z$};
\draw (1) edge[->,"$a$"] (2);
\draw (3) edge[->,"$b$"] (1);
\draw (2) edge[->,"$c$"] (3);
\end{scope}    
\end{tikzpicture}
\end{center}

By hypothesis of non-Louise type, the mutations at $x,y,z$ should still induce oriented cycles. By formula (\ref{mutation1}),(\ref{mutation2}), we have $ab>c,bc>a,ca>b$. In particular, we have $a,b,c\ge2$.  So the mutation relations, by definition of the form $xx'=y^a+z^b$,  do not have constant and linear terms. This argument holds for any seed, so during the mutation, the graph is always a oriented triangle with all edges having weight at least 2, and mutation relations do not have constatn and linear terms in terms of the cluster variables involved. By definition, the corresponding cluster algebra is the quotient of the polynomial ring with all cluster variables appear in the mutation process by the ideal of the mutation relations, we conclude that the degree 1 part is infinite dimensional $\BC$-vector spaces, and hence $A$ is not of finite type.
\end{proof}

Now we may assume the Louise property, \emph{i.e.} after suitable mutations, the graph is not an oriented cycle. In 3 vertices situation, there is only one possibility:

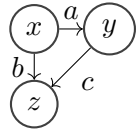
\begin{center}
\begin{tikzpicture}
\tikzstyle{round}=[circle,thick,draw=black!75, minimum size=4mm];
\tikzstyle{square}=[rectangle, thick,draw=black!75, minimum size=6mm];
\begin{scope}
\node[round] (1) {$x$};
\node[round](2)[right of =1]{$y$};
\node[round](3)[below of =1]{$z$};
\draw (1) edge[->,"$a$"] (2);
\draw (1) edge[->,"$b$",swap] (3);
\draw (2) edge[->,"$c$"] (3);
\end{scope}    
\end{tikzpicture}
\end{center}
or equivalently, the corresponding extended exchange matrix is of the form
\[
\left(
\begin{array}{ccc}
0&-a & -b\\
a&0 & -c\\
b&c & 0\\
\end{array}
\right)
\]
where $a,b,c\ge0$. When exactly one of $a,b,c$ equals 0, the cluster variety is singular. When at least two of $a,b,c$ equals zero, there exists isolated vertex which has no edges, which is isomorphic to case of Proposition \ref{2m} when $b=c=0$. Therefore, we assume $a,b,c>0$ in this section. We remark that the extended exchange matrix is not of full rank whatever $a,b,c$ are.

\begin{prop} \label{3m}
Let $\CA$ be the cluster variety associated with 
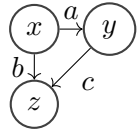
\begin{center}
\begin{tikzpicture}
\tikzstyle{round}=[circle,thick,draw=black!75, minimum size=4mm];
\tikzstyle{square}=[rectangle, thick,draw=black!75, minimum size=6mm];
\begin{scope}
\node[round] (1) {$x$};
\node[round](2)[right of =1]{$y$};
\node[round](3)[below of =1]{$z$};
\draw (1) edge[->,"$a$"] (2);
\draw (1) edge[->,"$b$",swap] (3);
\draw (2) edge[->,"$c$"] (3);
\end{scope}    
\end{tikzpicture}
\end{center}
where $a,b,c$ are positive integers. Then the mixed Hodge structure of $\CA$ is of mixed Tate type and the mixed Hodge numbers $h^{k,(p,p)}(\CA)$ are
\begin{center}
\begin{tabular}{c|cccc}
$k-p$ & $H^0$& $H^1$& $H^2$&$H^3$ \\
\hline
$0$ & $1$ & $0$ & $1$ & $1$\\
$1$ & $0$ &$0$& $C$& $C+1$ \\
\end{tabular}
\end{center}
where $C= \gcd(a, b) +\gcd(a,c) +\gcd(b,c) - 3$. 
\end{prop}

\begin{proof}
We apply the Mayer-Vietoris argument to the covering induced by the Louise algorithm of the separating edge $y\rightarrow z$. Then we have $\CA=U\cup V$, where $U$, $V$ and $U\cap V$ are the cluster varieties associated with the graphs

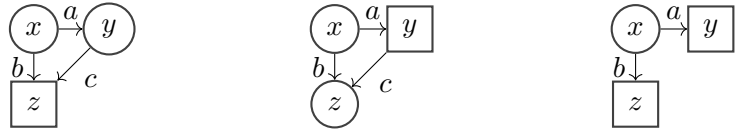
\begin{center}
\begin{tikzpicture}
\tikzstyle{round}=[circle,thick,draw=black!75, minimum size=4mm];
\tikzstyle{square}=[rectangle, thick,draw=black!75, minimum size=6mm];
\begin{scope}
\node[round] (1) {$x$};
\node[round](2)[right of =1]{$y$};
\node[square](3)[below of =1]{$z$};
\draw (1) edge[->,"$a$"] (2);
\draw (1) edge[->,"$b$",swap] (3);
\draw (2) edge[->,"$c$"] (3);

\node[round] at (4,0) (1) {$x$};
\node[square](2)[right of =1]{$y$};
\node[round](3)[below of =1]{$z$};
\draw (1) edge[->,"$a$"] (2);
\draw (1) edge[->,"$b$",swap] (3);
\draw (2) edge[->,"$c$"] (3);

\node[round] at (8,0) (1) {$x$};
\node[square](2)[right of =1]{$y$};
\node[square](3)[below of =1]{$z$};
\draw (1) edge[->,"$a$"] (2);
\draw (1) edge[->,"$b$",swap] (3);
\end{scope}    
\end{tikzpicture}
\end{center}

It suffices to find the kernel and the cokernel of 
\[
f^{k,(p,p)}:H^{k,(p,p)}(U)\oplus H^{k,(p,p)}(V)\to H^{k,(p,p)}(U\cap V)
\]
for all $k,p$ and apply Lemma \ref{MVH}. Note that the mixed Hodge structures of $U$, $V$ and $U\cap V$ are computed in Proposition \ref{1m} and Proposition \ref{2m}. We will always follow the convention in Section \ref{notation}.
\begin{enumerate}
\item $k=p=0$.  $\ker f^{0,(0,0)}=\langle1\rangle$ and $\textup{coker}\, f^{0,(0,0)}=0$.
\item $k=p=1$. We have
\[
\begin{cases}
H^{1,(1,1)}(U)=\langle \underline{z}\rangle & \textup{Proposition \ref{2m}}\\
H^{1,(1,1)}(V)=\langle \underline{y}\rangle & \textup{Proposition \ref{2m}}\\
H^{1,(1,1)}(U\cap V)=\langle \underline{y},\underline{z}\rangle & \textup{Proposition \ref{1m}}
\end{cases}
\]
So $\ker f^{1,(1,1)}= 0$ and $\textup{coker}\, f^{1,(1,1)}=0$.

\item $k=p=2$. We have
\[
\begin{cases}
H^{2,(2,2)}(U)=\langle a\underline{xy}+b\underline{xz}+c\underline{yz}\rangle & \textup{Proposition \ref{2m}}\\
H^{2,(2,2)}(V)=\langle b\underline{xz}-c\underline{zy}+a\underline{xy}\rangle=\langle a\underline{xy}+b\underline{xz}+c\underline{yz}\rangle & \textup{Proposition \ref{2m}}\\
H^{2,(2,2)}(U\cap V)=\langle  a\underline{xy} + b\underline{xz},\underline{yz}\rangle & \textup{Proposition \ref{1m}}
\end{cases}
\]
So $f^{2,(2,2)}$ has 1 dimensional image and hence $\dim \ker f^{2,(2,2)}=1$ and $\dim \textup{coker} \, f^{2,(2,2)}=1$.

\item $k=p=3$. $\ker f^{3,(3,3)}=\langle\underline{xyz}\rangle$ and $\textup{coker }f^{3,(3,3)}=0$.
\item $k=2,p=1.$ By Proposition \ref{1m} and Proposition \ref{2m}, we have
\[
\begin{cases}
H^{2,(1,1)}(U\cap V)& = \langle y^{\frac{i a}{(a, b)}}z^{\frac{ib}{(a, b)}}(a\underline{xy}+b\underline{xz}) \mid 1\le i\le (a,b)-1\rangle\\
H^{2,(1,1)}(U)=&\langle x^{\frac{i a}{(a,c)}}z^{\frac{ic}{(a,c)}}(a\underline{xy}+c\underline{yz}),y^{\frac{ja}{(a,b)}}z^{\frac{jb}{(a,b)}}(a\underline{xy}+b\underline{xz}),\\
&\mid 1\le i\le (a,c)-1,1\le j\le (b,c)-1\rangle
\end{cases}
\]
So the restriction map $H^{2,(1,1)}(U)\to H^{2,(1,1)}(U\cap V)$ is surjective. Therefore, $\textup{coker}\, f^{2,(1,1)}=0$ and 
\[
\begin{split}
\dim \ker f^{2,(1,1)}=&\dim H^{2,(1,1)}(U)+\dim H^{2,(1,1)}(V)-\dim H^{2,(1,1)}(U\cap V)\\
=&(a,b)+(b,c)+(c,a)-3.
\end{split}
\]

\item Similarly, $H^{3,(2,2)}(U)\to H^{3,(2,2)}(U\cap V)$ is surjective. We have $\textup{coker}\, f^{3,(2,2)}=0$ and $\dim \ker f^{2,(1,1)}=(a,b)+(b,c)+(c,a)-3$.
\end{enumerate}
In summary, the Hodge numbers are
\begin{center}
\begin{tabular}{c|cccc}
$k-p$ & $H^0$& $H^1$& $H^2$&$H^3$ \\
\hline
0 & 1 &0&1&1\\
1& 0 &0&$C$& $C+1$ 
\end{tabular}
\end{center}
where $C = (a,b)+(b,c)+(c,a)-3$.
\end{proof}

\begin{rmk}
\begin{enumerate}
\item We see from the proof that when $a,b,c$ are non-zero, the cluster variety $X$ is still smooth since it is covered by full rank cluster varieties $U$ and $V$. Even in this case, the mixed Hodge numbers fail to satisfy any numerical hard Lefschetz symmetry.
\item Most of the classes in the Deligne splitting can be represented explicitly by algebraic differential forms as Proposition \ref{1m} and Proposition \ref{2m}. The only exception is the class contributed by $\textup{coker}\, f^{2,(2,2)}$, which involves the boundary map in the Mayer-Vietoris sequence.
\end{enumerate}
\end{rmk}

\section{Towards singular cluster varieties}
There are no known systematic approaches to study the cohomology groups of singular rank cluster varieties, even for isolated cluster varieties. Based on analysis in Section \ref{smooth}, there are two types of non-full rank cluster varieties in dimension 3: 

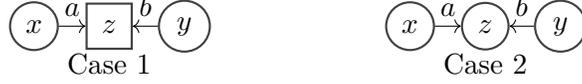
\begin{center}
\begin{tikzpicture}
\tikzstyle{round}=[circle,thick,draw=black!75, minimum size=4mm];
\tikzstyle{square}=[rectangle, thick,draw=black!75, minimum size=6mm];
\begin{scope}
\node[square] (1) {$z$};
\node[round] (2) [left of = 1]{$x$};
\node[round] (3) [right of = 1]{$y$};
\draw (2) edge[->,"$a$"] (1);
\draw (3) edge[->,"$b$",swap] (1);
\node at (0,-0.5) {Case 1};
\node[round] (1)  at (5,0) {$z$};
\node[round] (2) [left of = 1]{$x$};
\node[round] (3) [right of = 1]{$y$};
\draw (2) edge[->,"$a$"] (1);
\draw (3) edge[->,"$b$",swap] (1);
\node at (5,-0.5) {Case 2};
\end{scope}
\end{tikzpicture}
\end{center}

By applying Mayer-Vietoris argument, Case 2 reduces to Case 1 immediately. In this section, we apply classical mixed Hodge theoretic methods, such as compactifications and resolutions of singularities, to study the mixed Hodge structures on the cohomology and intersection cohomology of the cluster varieties $\CA$ of Case 1 when $a=b=1$, \emph{i.e.} defined by the graph
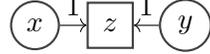
\begin{center}
\begin{tikzpicture}
\tikzstyle{round}=[circle,thick,draw=black!75, minimum size=4mm];
\tikzstyle{square}=[rectangle, thick,draw=black!75, minimum size=6mm];

\begin{scope}
\node[square] (1) {$z$};
\node[round] (2) [left of = 1]{$x$};
\node[round] (3) [right of = 1]{$y$};
\draw (2) edge[->,"1"] (1);
\draw (3) edge[->,"1",swap] (1);
\end{scope}
\end{tikzpicture}
\end{center}
or equivalently, the extended exchange matrix 
\[
\left(
\begin{array}{cc}
0&0\\
0&0\\
1&1\\
\end{array}
\right).
\]
By definition, the cluster variety $\CA$ is defined by equations
\[
xx'=yy'=z+1,z\neq0
\]
in $\BC^5$. 

We first calculate the cohomology groups of $\CA$. Let $\overline{\CA}$ be the partial compactification 
\[
\overline{\CA}=\{(x,x',y,y',z)\mid xx'=yy'=z+1\}.
\]
Then $\overline{\CA}$ is an affine cone and hence is contractible. Let $U$ be a tubular neighborhood of the boundary 
\[
U=\{(x,x',y,y',z)\mid xx'=yy'=z+1, \,|z|\le 1/2\}.
\]
Then the natural projection $\pi:U\to \Delta=\{|z|\le 1/2\}$ mapping $(x,x',y,y',z)$ to $z$ is a topological trivial fiber bundle. Then $U$ is diffeomorphic to $\pi^{-1}(0)\cong\BC^*\times \BC^*$ and $U\cap\CA\cong (\BC^*)^3$. Now apply the Mayer-Vietoris sequence to the covering $\overline{\CA}=\CA\cup U$, we have
\begin{align*}
0\to \BQ&\to H^0(\CA)\oplus \BQ     \to \BQ\\
 \to 0  &\to H^1(\CA)\oplus \BQ^2  \to \BQ^3\\
 \to 0  &\to H^2(\CA)\oplus \BQ     \to \BQ^3 \\
 \to 0  &\to H^3(\CA)               \to \BQ\\
 \to 0  &\to H^4(\CA)               \to 0.\\
\end{align*} 
So we conclude that 
\begin{equation} \label{coh}
H^k(\CA)=
\begin{cases}
\BQ& k=0,1,3,\\
\BQ^2& k=2,\\
0&\textup{otherwise.}
\end{cases}
\end{equation}

To calculate the mixed Hodge structure of $\CA$, we need a better compacfication. Consider the projection 
\begin{eqnarray*}
\BC^5&\to&\BC^4\\
(x,x',y,y',z)&\mapsto&(x,x',y,y'),
\end{eqnarray*}
It is straightforward to check that $\CA$ maps isomorphically to its image 
\[
X=\{(x,x',y,y')\mid xx'=yy'\neq 1\}
\] 
in $\BC^4$. Let
\begin{equation} \label{Xbar}
\overline{X}=\{[X_0:X_1:X_2:X_3:X_4]\mid X_0X_1=X_2X_3\}
\end{equation}
be the natural compactification of $X$ in $\BP^4$. Then the boundary divisor $D$ consists of two components, the divisor at infinite 
\begin{equation}  \label{D1}
D_1=\{[X_0:X_1:X_2:X_3:X_4]\mid X_0X_1=X_2X_3,\,X_4=0\}
\end{equation}
and the divisor corresponds to $xx'=yy'=1$
\begin{equation} \label{D2}
D_2=\{[X_0:X_1:X_2:X_3:X_4]\mid X_0X_1=X_2X_3=X_4^2\}.
\end{equation}

We first study the geometry of $D_1$ and $D_2$.
\begin{prop} \label{geom}
Let $\overline{X}$, $D_1$, and $D_2$ as above. Let $P_i$ be the point in $\BP^4$ whose $i$-th homogeneous coordinate is 1, and 0 otherwise.
\begin{enumerate}
\item $\overline{X}$ is the projective cone over $D_1$ with vertex $P_4$, and $D_1$ is isomorphic to $\BP^1\times\BP^1$.
\item $D_2$ has four singularities at $P_0$, $P_1$, $P_2$, $P_3$, all of which are of type $A_1$.
\item The intersection $D_1\cap D_2$ is the union of four lines: $\overline {P_0P_2}$, $\overline {P_2P_1}$, $\overline {P_1P_3}$ and $\overline {P_3P_0}$, and is transverse away from $P_0$, $P_1$, $P_2$, and $P_3$.
\end{enumerate}
\end{prop}

\begin{proof}
\begin{enumerate}
\item  By comparing the defining equations (\ref{Xbar}) and (\ref{D1}), it is obvious that $\overline{X}$ is the projective cone over $D_1$. The equation (\ref{D1}) also implies that $D_1$ is the Segre embedding of $\BP^1\times \BP^1$ into $\{X_4=0\}\cong\BP^3$. 
\item Let $f=X_0X_1-X_2X_3$ and $g=X_0X_1-X_4^2$. Then $D_2=\{f=g=0\}$.  To find the singularities of $D_2$, it suffice to find the points on $X_0X_1=X_2X_3=X_4^2$ at which the Jacobian matrix 
\[
\frac{\partial(f,g)}{\partial(X_0,\cdots,X_4)}=\left(\begin{array}{ccccc}
X_1&X_0&X_3&X_2&0\\
X_1&X_0&0&0&2X_4
\end{array}\right)
\]
is of rank 1, or equivalently, the two rows are proportional. Simple algebra shows that there are 4 singularities $P_0,P_1,P_2,P_3$. Since the defining equation $D_2$ is symmetric, it suffices to show that $P_1$ is an $A_1$ singularity. By using the canonical affine coordinates around $P_1$, we have $D_2$ is locally defined by $x_1=x_2x_3=x_4^2$ in $\BC^4$. Note that the projection to the hyperplane $x_1=0$ induces an isomorphism to $x_2x_3=x_4^2$ in $\BC^3$. A change-of-variable argument shows that the surface singularity $(0,0,0)$ of $x_2x_3=x_4^2$ is of type $A_1$. 
\item The intersection of $D_1\cap D_2$ is defined by 
\[
X_0X_1=X_2X_3=X_4=0.
\]
So one of $X_0$ and $X_1$ has to be 0, and one of $X_2$ and $X_3$ is 0. Therefore, the intersection of $D_1$ and $D_2$ consists of four lines $\overline{P_0P_2}$, $\overline{P_2P_1}$, $\overline{P_1P_3}$ and $\overline{P_3P_0}$. They form a configuration of a square. Without lose of generality, it suffices to check transversality at a general point on the line $\overline{P_0P_2}$, \emph{i.e.} $Q=[X_0:0:X_2:0:0]$ where $X_0$ and $X_2$ are nonzero. Then on the affine chart $X_0\neq0$, $D_1$ is defined by $x_1=x_2x_3$ and $x_4=0$, so the tangent space at a point $(0,a,0,0)$ is $x_1=ax_3$ and $x_4=0$. $D_2$ is defined by $x_1=x_2x_3=x_4^2$. so the tangent space at $(0,a,0,0)$ is $x_1=ax_3$ and $x_1=0$. So the intersection of $D_1$ and $D_2$ is transverse at $Q$.
\end{enumerate} 
\end{proof}

For the purpose of calculating the mixed Hodge numbers, we have to properly modify $\overline{X}$ such that the total space is smooth and the boundary divisor is simple normal crossing. To achieve this, we blow up five points $P_0,\cdots,P_4$ in $\overline{X}$. We have

\begin{prop} \label{compact}
   Let $\widetilde{X}=Bl_{P_0,\cdots,P_4}\overline{X}$. Let $\widetilde{D}_1$ and $\widetilde{D}_2$ be the strict transforms of $D_1$ and $D_2$ respectively. Let $E_i$ be the exceptional divisor over $P_i$, $0\le i\le 4$. Then we have the following.
   \begin{enumerate}
       \item The blow up $Bl_{P_4}\overline{X}$ is a $\BP^1$-bundle over $D_1\cong\BP^1\times \BP^1$. So the cohomology groups of $\widetilde{X}$ is 
       \[
       H^k(\widetilde{X})=
       \begin{cases} 
       \BQ & k=0,6,\\
       \BQ^7 &k=2,4,\\
       0 & \textrm{otherwise.}
       \end{cases}
       \]
       \item For $i=1,2$, the strict transforms $\widetilde{D_i}=Bl_{P_0,\cdots,P_3}D_i$ are smooth and the cohomology groups are
       \[
       H^k(\widetilde{D}_i)=
       \begin{cases} 
       \BQ & k=0,4,\\
       \BQ^6 &k=2,\\
       0 & \textrm{otherwise}.
       \end{cases}
       \]
       \item $E_i\cong \BP^2$ for $0\le i\le 3$ and $E_4\cong \BP^1\times \BP^1$.
       \item $\widetilde{D}_1\cap\widetilde{D}_2$ is isomorphic to the disjoint union of $4$ copies of $\BP^1$.
       \item $\widetilde{D}_i\cap E_j$ is isomorphic to $\BP^1$ for $i=0,1$ and $0\le j\le 3$.
       \item Triple intersection $\widetilde{D}_1\cap \widetilde{D}_2\cap E_j$ consists of two points for $0\le j\le 3$ and there are no further intersections.
   \end{enumerate}
   In particular, $\widetilde{X}$ and $E_4$ are smooth projective varieties, the boundary divisor  $\widetilde{D}:=\widetilde{D}_1+\widetilde{D}_2+\sum_{i=0}^3E_i$ is simple normal crossing, and hence
   \begin{equation} \label{resol}
   \begin{tikzcd}
      (E_4,\varnothing)\arrow[r]\arrow[d] & (\widetilde{X},\widetilde{D})\arrow[d]\\
      (\{P_4\},\varnothing)\arrow[r] & (\overline{X}, D)
   \end{tikzcd}
   \end{equation}
   is a 2-cubical log resolution.
\end{prop} 

\begin{proof}
\begin{enumerate}
\item By Proposition \ref{geom}.(1), blowing up the vertex $P_4$ of the projective cone $X$ over $D_1$ yields a $\BP^1$-bundle over $D_1$. Since $D_1\cong \BP^1\times \BP^1$ is simply connected, all local systems on $D_1$ are trivial. So by Deligne's decomposition theorem for the proper smooth morphism $\pi:Bl_{P_4}\overline{X}\to D_1$, we have
\[
\begin{split}
R\pi_*\BQ_{Bl_{P_4}\overline{X}}=&R^0\pi_*\BQ_{Bl_{P_4}\overline{X}}\oplus R^1\pi_*\BQ_{Bl_{P_4}\overline{X}}[-1]\oplus R^2f_*\BQ_{Bl_{P_4}\overline{X}}[-2]\\
=& \BQ_{D_1}\oplus \BQ_{D_1}[-2].
\end{split}
\]
The cohomology groups
\[
       H^k(Bl_{P_4}\overline{X})=
       \begin{cases} 
       \BQ & k=0,6,\\
       \BQ^3 &k=2,4,\\
       0 & \textrm{otherwise.}
       \end{cases}
       \]
Further blowing up smooth points $P_0,\cdots,P_3$, we have
    \[
       H^k(\widetilde{X})=
       \begin{cases} 
       \BQ & k=0,6,\\
       \BQ^7 &k=2,4,\\
       0 & \textrm{otherwise.}
       \end{cases}
       \]
  \item Since $\widetilde{D}_1$ is the blow up of $D_1\cong \BP^1\times \BP^1$ at 4 smooth points, 
       \[
       H^k(\widetilde{D}_1)=
       \begin{cases} 
       \BQ & k=0,4,\\
       \BQ^6 &k=2,\\
       0 & \textrm{otherwise}.
       \end{cases}
       \]
       By definition $D_2$ is a $(2,2)$-complete intersection in $\BP^4$ with four $A_1$ singularities. Since $A_1$ surface singularity has Milnor number 1, we have 
       \begin{equation} \label{euler}
       \chi(D_2)=\chi(D_0)-\sum_{i=0}^3 \mu(D_2,P_i)=8-4=4,
       \end{equation}
       where $D_0$ is any smooth $(2,2)$-complete intersection in $\BP^4$ and $\mu(D_2,P_i)$ are Milnor numbers; see \cite[(5.3.7.iv), (5.4.4.ii)]{D}. By \cite[(5.4.3.i)]{D}, we have $H^i(D_2)=H^i(\BP^2)$ for $i=0,1,4$. Since the Milnor lattice of $A_1$ singularity is $\BZ$, \cite[(5.4.6.A)]{D} implies that $H^3(D_2)=0$. So the only remaining piece $H^2(D_2)=\BQ^2$ follows from (\ref{euler}). Since blowing up an $A_1$ singularity increase the 2nd Betti number by 1, so
       \[
       H^k(\widetilde{D}_2)=
       \begin{cases} 
       \BQ & k=0,4,\\
       \BQ^6 &k=2,\\
       0 & \textrm{otherwise}.
       \end{cases}
       \]
   \item $E_0,\cdots,E_3$ are the exceptional divisors of the blow up smooth points in a threefold, so $E_0,\cdots,E_3\cong\BP^2$. By Proposition \ref{geom}.(1), $E_4$ is the exceptional divisor of blowing up the vertex of the cone over $\BP^1\times\BP^1$. So $E_4\cong \BP^1\times\BP^1$.
   \item By Proposition \ref{geom}.(3), the intersection $D_1\cap D_2$ before blow up is the union of four lines $\overline {P_0P_2}, \overline {P_2P_1},\overline {P_1P_3},\overline {P_3P_0}$. 
   Since at the intersections they have different slope, the blow-up at $P_0,\cdots, P_3$ just separate the branches and get 4 copies of $\BP^1$. Note that $\widetilde{D}_1\cap \widetilde{D}_2$ does not contain any curve in the exceptional divisors, so $\widetilde{D}_1\cap \widetilde{D}_2$ is just a disjoint union of 4 copies of $\BP^1$.
   \item By the construction of the blowup, a point $Q$ in the exception divisor $E_j$ represents a tangent direction of $\overline{X}$, and $Q\in \widetilde{D}_j$ if the direction it represents lies in $D_i$. Since $D_1$ is smooth at $P_j$ and $D_2$ has an $A_1$ singularity at $P_j$, all possible directions form a $\BP^1$ in both cases. We have $\widetilde{D}_i\cap E_j\cong \BP^1$. 
   \item The triple intersection $\widetilde{D}_1\cap \widetilde{D}_2\cap E_j$ consists two points which represent the directions of the two lines of $D_1\cap D_2$ passing through $P_j$. 
\end{enumerate}
\end{proof}

\begin{thm} \label{h}
   Consider the cluster variety associated with the directed diagram
   
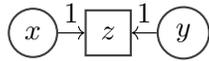
\begin{center}
\begin{tikzpicture}
\tikzstyle{round}=[circle,thick,draw=black!75, minimum size=4mm];
\tikzstyle{square}=[rectangle, thick,draw=black!75, minimum size=6mm];

\begin{scope}
\node[square] (1) {$z$};
\node[round] (2) [left of = 1]{$x$};
\node[round] (3) [right of = 1]{$y$};
\draw (2) edge[->,"$1$"] (1);
\draw (3) edge[->,"$1$",swap] (1);
\end{scope}
\end{tikzpicture}
\end{center}
Then its mixed Hodge numbers $h^{k,(p,p)}$ are
\begin{center}
\begin{tabular}{c|cccc}
$k-p$ & $H^0$& $H^1$& $H^2$&$H^3$ \\
\hline
$0$ & $1$ &$1$&$2$&$1$\\
\end{tabular}
\end{center}
In particular, the mixed Hodge structure is of mixed Tate type but not satisfy the numerical curious hard Lefschetz property.
\end{thm}

\begin{proof}
  Following \cite[7.1.14]{EL}, the 2-cubical resolution (\ref{resol}) yields a 2nd quadrant $E_1$ spectral sequence which converges to the mixed Hodge structure of $X$
  \[
  E_1^{p,q}= H^{q+2p}\left(\widetilde{D}_{\widetilde{X}}^{(-p)}\sqcup\varnothing_{P_4}^{(-p)}\right)\oplus H^{q+2p-2}\left(\varnothing_{E_4}^{(-p+1)}\right)  \Longrightarrow H^{p+q}(X),
  \]
  where the notation $Y_Z^{(k)}$ for a possibly empty simple normal crossing divisor $Y$ in a smooth variety $Z$ means the disjoint union of all $k$-fold intersections of the components of $Y$ for $k>0$ and $Y^{(0)}_Z=Z$. In our situation, $\varnothing_{P_4}^{(-p)}$ is the point $P_4$ if $p=0$ and empty otherwise. Similarly $\varnothing_{E_4}^{(-p+1)}$ is $E_4$ for $p=1$ and empty otherwise. The geometry of $\widetilde{D}_{\widetilde{X}}^{(-p)}$ is completely described in Proposition \ref{compact}. Combining all the results, the $E_1$ page of the spectral sequence is
  \[
   \begin{tabular}{c|c|c||c||c}
       8&12 &6 &1& \\
       \hline
        &0 & 0&0& \\
        \hline
        &12 & 16&7&1 \\
        \hline
        & & 0&0&0\\
        \hline
        & & 6&7&2\\
        \hline
        & & &0&0 \\
        \hhline{=|=|=||=||=}
        & & &2&1 \\
        \hhline{=|=|=||=||=}
   \end{tabular}
   \]
where the column and the row between the double lines corresponds to $p=0$ and $q=0$, and the numbers are the dimensions of the corresponding cohomology groups. Since this spectral sequence degenerates at $E_2$ page \cite[Proposition 7.13]{EL}, it suffices to study the differential $d_1$, which are by definition alternating sums of Gysin maps for $p\le -1$ and restriction maps for $p\ge0$. We simply omit the subscript of $d_1$ and just call it $d$.
  
Row $q=2$. By the proof of Proposition \ref{compact}.(1), $\widetilde{X}$ is the blow up of $\BP^1$-bundle of $E_4=l_1\times l_2$, the product of two lines, at four points. So $H^2(\widetilde{X})$ is spanned by basis $E_0,\dots,E_4, \BP^1_{l_1},\BP^1_{l_2}$, where $\BP^1_{l_i}$ denotes the $\BP^1$-bundle over the line $l_i$. The differential $d^{0,2}:E_1^{0,2}\to E_1^{1,2}$ is exactly the restriction $H^2(\widetilde{X})\to H^2(E_4)$, so $d^{0,2}$ is surjective and $\ker d^{0,2}=\langle E_0,\dots,E_4\rangle$ is 5-dimensional. On the other hand, $E_1^{-1,2}$ are freely spanned by the fundamental classes of $\widetilde{D}_1, \widetilde{D}_2, E_0,\cdots, E_3$ in $\widetilde{X}$, so $\textup{im}\,d^{-1,2}$ are the subspace in $H^2(X)$ spanned by these classes. Since $\widetilde{D_1}$ is the strict transform of $D_1\in \Pic(\overline{X})$ and $E_0,\dots,E_3$ are exceptional classes of the blowup $\widetilde{X}\to\overline{X}$, they are linearly independent, so $\dim \textup{im}\,d^{-1,2}\ge 5$. Therefore, we have $\ker d^{0,2}=\textup{im}\,d^{-1,2}$ and hence $E_2^{-1,2}\cong\BQ$ and $E_2^{0,2}=E_2^{1,2}=0$.

Row $q=4$.  Since $H^{\ge 4}(X)=0$ and the spectral sequence degenerates at $E_2$ page, the differentials are exact at $E_1^{0,4}$ and $E_1^{1,4}$. So $\dim \textup{im}\, d^{0,4}=1$ implies $\dim \textup{im}\, d^{-1,4}=\dim \ker d^{0,4}=6$, and hence $\dim \ker d^{-1,4}=16-6=10$. On the other hand, $\BQ^2\cong H^2(X)=E_2^{-2,4}\oplus E_2^{0,2}=\ker d^{-2,4}$. So $\textup{im}\,d^{-2,4}\cong \BQ^{10}$ and the spectral sequence is exact at $E^{-1,4}$. Therefore, $E_2^{-2,4}=\BQ^2$ and $E_2^{-1,4}=E_2^{0,4}=E_2^{1,4}=0$

Row $q=6$.  Since $H^{\ge 4}(X)=0$, the spectral sequence is exact at $E_1^{-2,6}$, $E_1^{-1,6}$ and $E_1^{0,6}$. Since the spectral sequence degenerates at $E_2$, we have 
\[
H^3(X)\cong E_2^{-3,6}\oplus E_2^{-1,4}\oplus E_2^{1,2}
\]
Now by (\ref{coh}) and $E_2^{1,2}=E_2^{-1,4}=0$, we have $E_2^{-3,6}\cong\BQ$.

Combining all results above, we have the $E_2$ page 
  \[
   \begin{tabular}{c|c|c||c||c}
       1&0 &0 &0& \\
       \hline
        &0 & 0&0& \\
        \hline
        &2 & 0&0 & 0\\
        \hline
        & & 0&0&0\\
        \hline
        & & 1&0&0\\
        \hline
        & & &0&0 \\
        \hhline{=|=|=||=||=}
        & & &1&0 \\
        \hhline{=|=|=||=||=}
   \end{tabular}
   \]
and hence the nonzero Deligne splittings are $\textup{Gr}_0H^0(X)=\BQ$, $\textup{Gr}_1H^1(X)=\BQ$, $\textup{Gr}_4H^2(X)=\BQ^2$, $\textup{Gr}_6H^3(X)=\BQ$.  Since the cohomology groups of all varieties which appear in the spectral sequence are of Tate type, the mixed Hodge structure of $X$ is of mixed Tate type, and $h^{k,(p,p)}$ satisfies
\begin{center}
\begin{tabular}{c|cccc}
$k-p$ & $H^0$& $H^1$& $H^2$&$H^3$ \\
\hline
$0$ & $1$ &$1$&$2$&$1$\\
\end{tabular}
\end{center}
\end{proof}

We can also calculate the mixed Hodge structure of the intersection cohomology of the singular cluster variety $X$.

\begin{thm} \label{ih}
   The mixed Hodge structure of the intersection cohomology $IH^k(X)$ is of mixed Hodge type and the mixed Hodge numbers are
\begin{center}
\begin{tabular}{c|cccc}
$k-p$ & $IH^0$& $IH^1$& $IH^2$&$IH^3$ \\
\hline
$0$ & $1$ &$1$&$2$&$1$\\
$1$ &     &   &$1$&   \\
\end{tabular}
\end{center}
\end{thm}

\begin{proof}
Let $X'=Bl_{P_4}X$ be the blow up of the singularity. By Proposition \ref{compact}.(1), $X'$ is smooth and 
\begin{equation} \label{disc}
\begin{tikzcd}
E_4\arrow[r]\arrow[d] & X'\arrow[d]\\
\{P_4\}\arrow[r]& X
\end{tikzcd}
\end{equation}
is a proper modification. Then by \cite[Corollary-Definition 5.37]{PS}, the Mayer-Vietoris sequence for the discriminant square (\ref{disc}) yields an exact sequence of mixed Hodge structures
\begin{align*}
0\to& \BQ\to H^0(X')\oplus \BQ \to \BQ\\
\to& \BQ(-1)\to H^1(X') \to 0\\
\to& \BQ^2(-2)\to H^2(X') \to \BQ^2(-1)\\
\to& \BQ(-3)\to H^3(X') \to 0\\
\to&0\to H^4(X') \to \BQ(-2)\to0,\\
\end{align*}
where $\BQ(k)$ is the Tate twist which is of Hodge type $(-k,-k)$. So the Hodge numbers  $h^{k,(p,p)}(X')$ are
\begin{center}
\begin{equation}\label{table}
\begin{tabular}{c|ccccc} 
$k-p$ & $H^0$& $H^1$& $H^2$&$H^3$ &$H^4$ \\
\hline
$0$ & $1$ &$1$&$2$&$1$& \\
$1$ &     &   &$2$&  & \\
$2$ &     &   & &  & $1$
\end{tabular}
\end{equation}
\end{center}

Apply the decomposition theorem to the natural morphism $p:X'=Bl_{P_4}X\to X$, we have
\begin{equation} \label{aaa0}
Rp_*\BQ_{X'}[3]\cong IC_X\oplus F
\end{equation}
where $F$ supports at the singular point $P_4$. To determine $F$, it suffices to compare the stalks at $P_4$. By proper base change theorem, 
\[
(R^kp_*\BQ_{X'}[3])_{P_4}=H^{k+3}\left(p^{-1}(P_4),\BQ\right)\cong H^{k+3}(\BP^1\times\BP^1,\BQ),
\]
so 
\begin{equation} \label{aaa1}
(Rp_*\BQ_{X'}[3])_{P_4}=\BQ_{P_4}[3]\oplus\BQ^2_{P_4}(-1)[1]\oplus\BQ_{P_4}(-2)[-1].
\end{equation}
By definition, $IC_X=\tau_{\le-1}(Rj_*\BQ_U[3])$, where $j:U=X\setminus\{P_4\}\hookrightarrow X$ is the open embedding. Since $(R^kj_*\BQ_U[3])_{P_4}=\varinjlim_{U\ni P_4} H^{k+3}(U\setminus \{P_4\},\BQ)$, we have 
\[
(Rj_*\BQ_U[3])_{P_4}=\BQ_{P_4}[3]\oplus\BQ_{P_4}(-1)[1]\oplus \BQ_{P_4}(-2)\oplus\BQ_{P_4}(-3)[-2]
\]
and hence
\begin{equation} \label{aaa2}
(IC_X)_{P_4}=\BQ_{P_4}[3]\oplus\BQ_{P_4}(-1)[1].
\end{equation}
Comparing (\ref{aaa0}), (\ref{aaa1}) and (\ref{aaa2}), we have $F=\BQ_{P_4}[1]\oplus \BQ_{P_4}[-1]$, and
\begin{equation} \label{bbb}
Rp_*\BQ_{X'}[3]\cong IC_X\oplus \BQ_{P_4}(-1)[1]\oplus \BQ_{P_4}(-2)[-1].
\end{equation}
Taking the hypercohomology of (\ref{bbb}), we have the isomorphism of Hodge structures
\begin{equation}\label{ccc}
H^*(X',\BQ)=IH^*(X)\oplus \BQ(-1)[-2]\oplus \BQ(-2)[-4].
\end{equation}
Compare with the table (\ref{table}), we conclude that the mixed Hodge number $h^{k,(p,p)}$ of the intersection cohomology $IH(X)$ is
\begin{center}
\begin{tabular}{c|cccc}
$k-p$ & $IH^0$& $IH^1$& $IH^2$&$IH^3$ \\
\hline
$0$ & $1$ &$1$&$2$&$1$\\
$1$ &     &   &$1$&   \\
\end{tabular}
\end{center}

\end{proof}

\end{document}